\documentclass[12pt]{amsart}
\usepackage{times}
\usepackage{amsmath,amsfonts,amssymb,amsopn,amscd,amsthm}
\usepackage{graphicx,xspace}
\usepackage{epsfig}
\usepackage{enumerate} 
\usepackage{enumitem}
\usepackage{xcolor}
\usepackage[all]{xypic}

\usepackage[T1]{fontenc}
\usepackage[utf8]{inputenc}

\usepackage{hyperref}
\usepackage{psfrag}
\usepackage{bm}
\usepackage{youngtab}
\usepackage{tikz}

\theoremstyle{plain}
\newtheorem{lemma}{Lemma}[section]

\theoremstyle{remark}

\newtheorem{definition}[lemma]{Definition}

\author[O. Tout]{Omar Tout}
\address{LaBRI, Universit\'e de Bordeaux, 351 cours de la Lib\'eration, 33 400
Talence, France}
\email{omar.tout@labri.fr}

\title[]{
A Frobenius formula for the structure coefficients of double-class algebras of Gelfand pairs}

\usetikzlibrary{snakes}
\usepackage[all]{xy}
\usepackage{varioref}
\theoremstyle{plain}
\newtheorem{theoreme}{Theorem}[section]
\newtheorem{prop}[theoreme]{Proposition}

\newtheorem{cor}[theoreme]{Corollary}

\theoremstyle{definition}
\theoremstyle{remark}

\newtheorem*{notation}{Notation}

\date{}

\DeclareMathOperator{\diag}{diag}
\DeclareMathOperator{\tr}{tr}

\DeclareMathOperator{\Res}{Res}
\DeclareMathOperator{\Ind}{Ind}
\begin{document}
\maketitle
\paragraph{Abstract.} We generalise some well known properties of irreducible characters of finite groups to zonal spherical functions of Gelfand pairs. This leads to a Frobenius formula for Gelfand pairs. For a given Gelfand pair, the structure coefficients of its associated double-class algebra can be written in terms of zonal spherical functions. This is a generalisation of the Frobenius formula which writes the structure coefficients of the center of a finite group algebra in terms of irreducible characters.

\section{Introduction.}

Given a finite dimensional algebra with a fixed basis, the product of two basis elements can be written as a linear combination of basis elements. The coefficients that appear in this expansion are called \textit{structure coefficients}. Computing these coefficients is essential because, by linearity, they allow us to compute all products in the considered algebra. However, it is difficult, even in particular cases of algebras, to give explicit formulas for these coefficients.

\bigskip

The structure coefficients of centers of group algebras and of double-class algebras had been the most studied in the literature. In fact, the study of the structure coefficients of centers of group algebras is related to the representation theory of finite groups by Frobenius theorem, see \cite{JaVi90}[Lemma 3.3] and the appendix of Zagier in \cite{lando2004graphs}. 
This theorem expresses, in the general case of centers of group algebras, the structure coefficients in terms of irreducible characters.

The case of the center of the symmetric group algebra $Z(\mathbb{C}[\mathcal{S}_{n}])$ is particularly interesting and many authors have studied it in details. To compute the structure coefficients of $Z(\mathbb{C}[\mathcal{S}_{n}])$ by a direct way one should study the cycle-type of the product of two permutations which can be quite difficult, see for example the papers \cite{bertram1980decomposing}, \cite{boccara1980nombre}, \cite{Stanley1981255}, \cite{walkup1979many}, \cite{GoupilSchaefferStructureCoef} and \cite{JaVi90} which deal with particular cases of these coefficients. 

According to another formula due to Frobenius, see \cite{sagan2001symmetric}, the irreducible characters of the symmetric group appear in the development of Schur functions in terms of power functions. This links the study of the structure coefficients of the center of the symmetric group algebra to that of the symmetric functions theory.

\bigskip

A pair $(G,K),$ where $G$ is a finite group and $K$ a sub-group of $G,$ is said to be a Gelfand pair if its associated double-class algebra $\mathbb{C}[K\setminus G/K]$ is commutative, see \cite[chapter VII.1]{McDo}. If $(G,K)$ is a Gelfand pair, the algebra of constant functions over the $K$-double-classes in $G$ has a particular basis, whose elements are called \textit{zonal spherical functions} of the pair $(G,K).$ 

For a finite group $G,$ the pair $(G\times G^{opp},\diag(G))$ is a Gelfand pair, see the Example 9 in \cite[Section VII.1]{McDo} or the introduction of Strahov in \cite{strahov2007generalized}. We study this pair in details in Sections \ref{sec:GelfandpairGxGopp} and \ref{sec:FrobeniustheorembyusingtheGelfandpair}. What makes this pair exceptional is that its zonal sphercial functions are the normalised irreducible characters of $G.$ This allows us to see the double-class algebras associated to Gelfand pairs (resp. zonal sphercial functions of Gelfand pairs) as generalisation of centers of group algebras (resp.  normalised irreducible characters).

The pair $(\mathcal{S}_{2n}, \mathcal{B}_n)$ is a Gelfand pair, see \cite[Section VII.2]{McDo}. Its associated double-class algebra has a long list of properties similar to the center of the symmetric group algebra, see \cite{Aker20122465}, \cite{2014arXiv1402.4615D} and \cite{toutejc}. The study of its structure coefficients is also related to the theory of symmetric functions since the zonal spherical functions of the pair $(\mathcal{S}_{2n}, \mathcal{B}_n)$ appear in the development of zonal polynomials in terms of power-sums. The zonal polynomials are specialization of Jack polynomials, defined by Jack in \cite{jack1970class} and \cite{jack1972xxv}, which form a basis for the algebra of symmetric functions. 

For special cases of Gelfand pairs such as the Gelfand pair of $(\mathcal{S}_{2n},\mathcal{B}_n)$ and that of $(\mathcal{S}_n\times \mathcal{S}_{n-1}^{opp},\diag(\mathcal{S}_{n-1}))$ there exists a Frobenius formula, given by Goulden and Jackson in \cite{GouldenJacksonLocallyOrientedMaps} and by Jackson and Sloss in \cite{jackson2012character} respectively, which writes the structure coefficients of the double-class algebras $\mathbb{C}[\mathcal{B}_{n}\setminus \mathcal{S}_{2n}/ \mathcal{B}_{n}]$ and $\mathbb{C}[\diag(\mathcal{S}_{n-1})\setminus \mathcal{S}_n\times \mathcal{S}^{opp}_{n-1}/ \diag(\mathcal{S}_{n-1})]$ in terms of zonal spherical functions. 

According to the author's knowledge, there isn't such a formula which treats the general case of Gelfand pairs. The goal of this paper is to give, for a Gelfand pair $(G,K),$ a formula similar to that of Frobenius which writes the structure coefficients of the double-class algebra $\mathbb{C}[K\setminus G/ K]$ in terms of zonal spherical functions.


This formula can help computing the structure coefficients of double-class algebras of Gelfand pairs when the direct calculation of these coefficients is difficult to conduct. It was already used in the cases of $\mathbb{C}[\mathcal{B}_{n}\setminus \mathcal{S}_{2n}/ \mathcal{B}_{n}]$ and $\mathbb{C}[\diag(\mathcal{S}_{n-1})\setminus \mathcal{S}_n\times \mathcal{S}^{opp}_{n-1}/ \diag(\mathcal{S}_{n-1})],$ see \cite{bernardi2011counting}, \cite{GouldenJacksonLocallyOrientedMaps}, \cite{jackson1990character} and \cite{JaVi90} for the first algebra and \cite{Jackson20121856} for the second. 

The majority of the results presented in this paper can be found in the author's Phd thesis \cite{touPhd14}. Theorem \ref{Th_prin_gener} and the application section, Section \ref{sec:app}, results can not be found there. The author's Phd thesis dealt primarily with the polynomiality property of the structure coefficients of some algebras and that study is where the results given in this paper come from. However, we decided to split the content of our Phd thesis in two papers since the results presented here do not concern the polynomiality property of the structure coefficients. The author is preparing another paper \cite{toutpoly} about the polynomiality property. We do not need the results in that paper for our work here but some results presented in this paper would be useful in \cite{toutpoly}.

\section{Structure coefficients in general}

In this section we define the structure coefficients in general. At the end we give a basic and simple proposition, Proposition \ref{desc_coef}, in particular cases of algebras which describes the structure coefficients. It will help us computing the structure coefficients in the next sections.

\begin{definition}
Let $\mathcal{K}$ be a commutative field and let $I$ be a finite set. Suppose that $\mathcal{B}$ is a finite dimensional algebra over $\mathcal{K}$ and that the elements of the family $(b_k)_{k\in I}$ form a basis for $\mathcal{B}.$ Take two basis elements $b_i$ and $b_j$ of $\mathcal{B}$ (here $i, j\in I$), then the product $b_ib_j$ is an element of $\mathcal{B},$ therefore it can be written as a linear combination of the basis elements $(b_k)_{k\in I} :$ 
\begin{equation}\label{coef_de_str}
b_ib_j=\sum_{k\in I}c_{ij}^kb_k
\end{equation}
where the coefficients $c_{ij}^k$\label{nomen:cij} are in $\mathcal{K}$ for every $k\in I.$ The coefficients $c_{ij}^k$ are called the {\em structure coefficients} of $\mathcal{B}$ in the basis $(b_k)_{k\in I}.$ If there is no confusion concerning the basis elements, we talk simply about structure coefficients without mentioning the basis.
\end{definition}

Note that all the algebras considered in this paper are over the field of complex number $\mathbb{C}.$ In some cases of algebras (like the ones considered in this paper), we can give a combinatorial description for the structure coefficients which is explained in the coming paragraph.

\begin{definition}
An algebra basis is said to be {\em multiplicative} if the product of two basis elements is a basis element.
\end{definition}
Suppose that we have a finite dimensional algebra $\mathcal{A}$ with a multiplicative basis $(b_k)_{k\in K},$ which we fix in this section, and a "type": $K \rightarrow I$ function where $I$ is a finite set. For an element $i\in I,$ define $B_i$ to be the set of all basis elements $b_k$ of $\mathcal{A}$ of type $i$ (the elements $b_k$ such that type($k$)=$i$). We can suppose that all the sets $B_i$ are non-empty.
\begin{notation}
If $X$ is a finite set, we denote by ${\bf X}$ the formal sum of its elements,
$$\textbf{X}=\sum_{x\in X}x.$$
\end{notation}

The elements of the family $({\bf B}_i)_{i\in I}$ are linearly independent. Suppose that the vector space $\mathcal{B}$ generated by the elements $({\bf B}_i)_{i\in I}$ forms a sub-algebra of $\mathcal{A}.$ In this case, there exists a combinatorial description for the structure coefficients of the algebra $\mathcal{B}$ given by the following proposition. This description will help us to compute the structure coefficients in the coming sections.
\begin{prop}\label{desc_coef}
Let $C_{ij}^k$ be the structure coefficients of the algebra $\mathcal{B},$ defined by the following equation:
\begin{equation}\label{coef_de_str}
{\bf B}_i{\bf B}_j=\sum_{k\in I}C_{ij}^k{\bf B}_k.
\end{equation}
Then $C_{ij}^k$ is the size of the following set:
\begin{multline}
C_{ij}^k=|\lbrace (x,y)\in \mathcal{A}^2 \text{ such that $x$ and $y$ are two basis elements of $\mathcal{A}$} \\
\text{of type $i$ and $j$ respectively and xy=z}\rbrace|,
\end{multline}
where $z\in \mathcal{A}$ is a fixed basis element of $\mathcal{A}$ of type $k.$ 
\end{prop}
\begin{proof}
This is obtained directly from the definition of the structure coefficients.
\end{proof}

Even if we consider algebras over the field $\mathbb{C}$ in this paper, we should mention that under the conditions of this proposition, we have the following corollary.

\begin{cor}
The structure coefficients of the algebra $\mathcal{B}$ are positive integers.
\end{cor} 

\section{Structure coefficients of centers of finite groups algebras and double-class algebras}

In this section we define the structure coefficients of centers of finite groups algebras and of double-class algebras. These coefficients will be the core of our study in the coming sections.

\subsection{Centers of finite groups algebras}

Let $G$ be a finite group. For an element $g$ of $G,$ we denote by $C_g$ its associated conjugacy class in $G$ :
$$C_g:=\lbrace g'gg'^{-1}~~\vert~~g'\in G\rbrace.$$
This size of $C_g$ is known to be equal to the following fraction:
$$|C_g|=\frac{|G|}{|Stab_g|},$$
where $Stab_g:=\lbrace g' \in G~~\vert~~g'g=gg'\rbrace.$

Since $G$ is finite, we have a finite number of conjugacy classes in $G.$ We denote $\mathcal{C}(G)$ the set of conjugacy classes of $G$ and we consider a finite set which we denote by $\mathcal{I}$ to index its elements. The elements of $\mathcal{I}$ will be written using Greek letters. Thus, we have : $$\mathcal{C}(G)=\lbrace C_\lambda \mid \lambda\in \mathcal{I}\rbrace.$$

The group algebra of $G,$ known as $\mathbb{C}[G],$ is the algebra over $\mathbb{C}$ with basis the elements of the group $G.$ 
The \textit{center of the group algebra} $\mathbb{C}[G]$, known as $Z(\mathbb{C}[G]),$ is the sub-algebra of $\mathbb{C}[G]$ of invariant elements under the conjugation action of $G$ on $\mathbb{C}[G]$ : 
$$Z(\mathbb{C}[G]):=\lbrace x\in \mathbb{C}[G] ~~|~~g'x=xg'~~\forall g'\in G\rbrace.$$

The conjugacy classes of $G$ index a basis of the center of the group algebra $Z(\mathbb{C}[G])$ as it is showed in the following proposition.
\begin{prop}\label{prop:base_centre_alg-de_grp} The family $({\bf C}_\lambda)_{\lambda\in \mathcal{I}}$, where
$${\bf C}_\lambda=\sum_{g\in C_\lambda}g,$$ 
form a basis for $Z(\mathbb{C}[G]).$
\end{prop}
\begin{proof}
Suppose that $X=\sum_{g\in G} c_gg$ is an element of the centre of the group $G$ algebra, then by definition, for any $g'\in G$ we have: 
$$\sum_{g\in G} c_gg=\sum_{g\in G} c_gg'gg'^{-1}.$$
This means that for any two elements $g$ and $g'$ of $G,$ the coefficients of $g$ and $g'gg'^{-1}$ in the expansion of $X$ are equals. Therefore, to be in $Z(\mathbb{C}[G]),$ the coefficients of any element $g$ in $X$ must be equal to the coefficient of any element conjugated to $g$ in $X.$ Thus, the sums of the elements of the conjugacy classes of $G$ form a basis (it is not difficult to see that these sums are linearly independent) for $Z(\mathbb{C}[G]).$ 
\end{proof}


Let $\lambda$ and $\delta$ be two elements of $\mathcal{I}.$ The structure coefficients $c_{\lambda\delta}^\rho$ of the center of the group algebra of $G$ are defined by the following equation :
\begin{equation}\label{coef_centre}
{\bf C}_\lambda{\bf C}_\delta=\sum_{\rho\in \mathcal{I}} c_{\lambda\delta}^\rho {\bf C}_g.
\end{equation}

\subsection{Double-class algebras}


Let $G$ be a finite group and let $H,$ $K$ and $L$ be three sub-groups of $G.$ A double-class of $H$ and $K$ in $G$ is a set $HgK$ for an element $g$ of $G$ where :
$$HgK:=\lbrace hgk \, ; \, h\in H \text{ and } k\in K\rbrace.$$
In the case where $H=K$ and in order to simplify we will use the notion of a double-class of $K$ in $G$ instead of a double-class of $K$ and $K$ in $G.$ The set $H\backslash G/K:=\lbrace HgK \, ; \, g\in G\rbrace$ is the set of double-classes of $H$ and $K$ in $G.$ We have the following well known proposition for the size of a double-class.

\begin{prop}If $H$ and $K$ are two sub-groups of a finite group $G,$ then for every $g\in G$ we have :
$$|HgK|=\frac{|H||K|}{|H\cap gKg^{-1}|}.$$
\end{prop}
\begin{proof}
Since $|hgK|=|K|$ for any $h\in H,$ we have :
$$|HgK|=|\lbrace hgK \text{ such that } h\in H\rbrace||K|.$$
If we consider the action of the sub-group $H$ on the set of the left classes of $K$ in $G,$ we get :
$$|HgK|=\frac{|H|}{|\lbrace h\in H \text{ such that }hgK=gK\rbrace|}|K|.$$
But $hgK=gK\Leftrightarrow h\in gKg^{-1},$ thus we get the result.
\end{proof}
 
\begin{prop}
Let $G$ be a finite group and $K$ a sub-group of $G.$ The sums of the double-classes $KgK$ of $K$ in $G$ linearly generate a sub-algebra of $\mathbb{C}[G].$
\end{prop}

\begin{proof}
In a more general case, if $H$ and $L$ are two sub-groups of $G$, by using the decomposition of double-classes into disjoint union of left and right classes, we can write :
$$\textbf{H}x\textbf{K}=\sum_i \textbf{H}x_i,$$
where the $x_i$'s are in the set $xK.$ For an another double-class $KyL$ of $K$ and $L$ in $G$, we can also write:
$$\textbf{K}y\textbf{L}=\sum_j y_j\textbf{L},$$
where the $y_j$'s are in the set $Ky.$
By using these two decompositions, the multiplication of two double-classes $\textbf{H}x\textbf{K}$ and $\textbf{K}y\textbf{L}$ can be written as follows :
\begin{eqnarray*}
{\textbf{H}} x {\textbf{K}}\textbf{K}y\textbf{L}&=& \sum_i \textbf{H}x_i \sum_j y_j\textbf{L}\\
&=&\sum_i\sum_j\textbf{H}x_iy_j\textbf{L}.
\end{eqnarray*}
We get our result in the particular case where $H=L=K.$
\end{proof}
 
The \textit{double-class algebra} of $K$ in $G,$ denoted by $\mathbb{C}[K\setminus G/K],$ is the algebra with basis the sums of elements of the double-classes of $K$ in $G.$ Since we consider $G$ as finite group, the cardinal of $H\backslash G/K$ is finite. We will use a finite set $\mathcal{J}$ to index the set $K\setminus G/K$ of double-classes of $K$ in $G.$ We will use also the Greek letters to represent the elements of $\mathcal{J}.$ Thus, we have :
$$K\setminus G/K=\lbrace DC_\lambda \mid \lambda\in \mathcal{J}\rbrace.$$
Let $\lambda$ and $\delta$ be two elements of $\mathcal{J}.$ The structure coefficients $k^\rho_{\lambda\delta}$ of the double-class algebra $\mathbb{C}[K\setminus G/K]$ are defined by the following equation :
\begin{equation}
\textbf{DC}_\lambda\textbf{DC}_\delta=\sum_{\rho \in \mathcal{J}}k^\rho_{\lambda\delta}\textbf{DC}_\rho.
\end{equation}

\section{irreducible characters and the Frobenius formula}\label{sec:irr_cha_fr_for}

The representation theory is a useful tool to describe the structure coefficients of the centers of finite group algebras. It is known that the number of irreducible $G$-modules is equal to the number of conjugacy classes of $G$ which is the dimension of the center $Z(\mathbb{C}[G])$ of the group algebra of $G$ as shown in Proposition \ref{prop:base_centre_alg-de_grp}. More than this, the Frobenius formula presented at the end of this section writes the structure coefficients of $Z(\mathbb{C}[G])$ in terms of irreducible characters of $G,$ which links the study of these structure coefficients to the representation theory of finite groups.

Throughout this section, we will give important results relating the representation theory of finite groups to that of structure coefficients. These results were already well known in the literature. We remember them here in order to give similar results for Gelfand pairs in the next sections.

We denote by $\hat{G}$ the set of irreducible $G$-modules :
$$\hat{G}:=\lbrace X \text{ such that $X$ is an irreducible $G$-module}\rbrace.$$
If $X$ is a $G$-module, we will denote $\mathcal{X}$ its character. For a conjugacy class $\mathcal{K}\in \mathcal{C}(G)$ and a character $\mathcal{X},$ we denote by $\mathcal{X}_{\mathcal{K}}$ the value of $\mathcal{X}$ on any element of $\mathcal{K},$ 
$$\mathcal{X}_{\mathcal{K}}:=\mathcal{X}(g) \text{ for every }g\in \mathcal{K}.$$

Using Maschke theorem, see \cite[Theorem 1.5.3]{sagan2001symmetric}, we have the following proposition.
\begin{prop}\label{dec_alg_de_grp}
Let $G$ be a finite group. Its associated group algebra $\mathbb{C}[G]$ can be decomposed as follows :
$$\mathbb{C}[G]=\bigoplus_{X\in \hat{G}}m_XX.$$
Moreover, we have :
\begin{enumerate}
\item[1-] $m_X=\dim X$ for every $X\in \hat{G}.$
\item[2-] $\dim \mathbb{C}[G]=|G|=\sum_{X\in \hat{G}}(\dim X)^2.$
\end{enumerate}
\end{prop}

Another interesting result in representation theory is the following :
\begin{prop}\label{egalite_classe_caract}
Let $G$ be a finite group. Then we have :
$$|\mathcal{C}(G)|=|\hat{G}|.$$
\end{prop}
\begin{proof}
For a complete proof of this proposition see the proof of Proposition 1.10.1 in \cite{sagan2001symmetric}.
\end{proof}

It follows from this Proposition and from Proposition \ref{prop:base_centre_alg-de_grp} that the dimension of the centre $Z(\mathbb{C}[G])$ of a finite group algebra is equal to the number of irreducible representations of $G.$ For the rest of this section we will give all important results which leads us to the Frobenius formula which writes the structure coefficients of $Z(\mathbb{C}[G])$ in terms of irreducible characters.
 
\begin{prop}\label{caract_norma}
Let $V$ be an irreducible $G$-module, the function $\frac{\mathcal{V}}{\dim V}$ defined by 
$$\begin{array}{ccccc}
\frac{\mathcal{V}}{\dim V} & : & Z(\mathbb{C}[G]) & \to & \mathbb{C} \\
& & x & \mapsto & \frac{\mathcal{V}(x)}{\dim V} \\
\end{array}$$
is a morphism of algebras.
\end{prop}
\begin{proof}
This proposition can be obtained using Schur lemma.
\end{proof}

This proposition can be otherwise written as follows.
\begin{prop}\label{isom_centre_fct}
Let $G$ be a finite group, the function 
$$\begin{array}{ccccc}
Z(\mathbb{C}[G]) & \to & \mathcal{F}(\hat{G},\mathbb{C}) \\
x & \mapsto & (V\mapsto \frac{\mathcal{V}(x)}{\dim V}) \\
\end{array}$$
is an isomorphism of algebras where $\mathcal{F}(\hat{G},\mathbb{C})$ is the set of functions defined on $\hat{G}$ with values in $\mathbb{C}.$
\end{prop}
\begin{proof}
The Proposition \ref{caract_norma} implies that this function is a morphism of algebras. It still to remark, using Proposition \ref{egalite_classe_caract}, that
\begin{equation*}
\dim Z(\mathbb{C}[G])=|\mathcal{C}(G)|=|\hat{G}|=\dim \mathcal{F}(\hat{G},\mathbb{C}). \qedhere
\end{equation*} 
\end{proof}


\begin{theoreme}[Frobenius]\label{coef_fct_cara}
Let $G$ be a finite group and let $\mathcal{I}$ be the finite set of elements which indexes the set of conjugacy classes of $G.$ Let $\lambda, \delta$ and $\rho$ be three elements of $\mathcal{I},$ the structure coefficient $c_{\lambda\delta}^\rho$ of the center of the group $G$ algebra can be written in terms of irreducible characters of $G$ as follows :
\begin{equation}
c_{\lambda\delta}^\rho=\frac{|C_\lambda||C_\delta|}{|G|}\sum_{X\in \hat{G}}\frac{\mathcal{X}_\lambda\mathcal{X}_\delta\overline{\mathcal{X}_\rho}}{\mathcal{X}(1_G)}.
\end{equation}
\end{theoreme}

The Frobenius formula given in Theorem \ref{coef_fct_cara} is a particular case of a generalised one which describes the product of more than two basis elements, see the appendix of Don Zagier in \cite{lando2004graphs}. This formula is given also in \cite[Lemma 3.3]{JaVi90}. It is used by many authors to compute the structure coefficients of the center of the symmetric group algebra.

\section{Structure coefficients of centers of group algebras and random variables} 

Using Proposition \ref{dec_alg_de_grp} of the last section, we can define a probability measure on the set $\hat{G}$ of irreducible $G$-modules of a finite group $G$ by the following manner. 

\begin{prop}
Let $G$ be a finite group. The function $\mathrm{P}_G:\hat{G}\rightarrow \mathbb{R}$ defined by :
$$\mathrm{P}_G(X)=\frac{(\dim X)^2}{|G|}$$
is a probability measure on $\hat{G}.$
\end{prop}
\begin{proof}
The fact that 
$$\sum_{X\in \hat{G}}\mathrm{P}_G(X)=\sum_{X\in \hat{G}}\frac{(\dim X)^2}{|G|}=1$$
is a consequence of Proposition \ref{dec_alg_de_grp} of the precedent section.
\end{proof}

The probability measure $\mathrm{P}_{G}$ is called the \textit{Plancherel measure}. Now fix an element $g$ of $G$ and consider the random variable $\mathfrak{F}_g$ defined as follows :
 
$$\begin{array}{ccccc}
\mathfrak{F}_g & : & \hat{G} & \to & \mathbb{R} \\
& & X & \mapsto & \frac{\mathcal{X}(g)}{\dim X} \\
\end{array}$$

The aim of this section is to study the random variable $\mathfrak{F}_g$ and to show the relation between the moments of this random variable and the structure coefficients of the center of the group $G$ algebra. It was very helpful for me to use the notes of a course given by Féray in Edinburgh, see \cite{feray2014}, in order to present this relation here.

The expectation $\mathbb{E}_{\mathrm{P}_{G}}(\mathfrak{F}_g)$ of the random variable $\mathfrak{F}_g$ is defined as follows :
$$\mathbb{E}_{\mathrm{P}_{G}}(\mathfrak{F}_g):=\sum_{X\in \hat{G}}\mathrm{P}_G(X)\mathfrak{F}_g(X).$$

\begin{prop}\label{esprerance}
Let $G$ be a finite group and let $g$ be a fixed element of $G,$ then we have :
$$\mathbb{E}_{\mathrm{P}_{G}}(\mathfrak{F}_g)=\delta_{\lbrace 1_G\rbrace}(g).$$
\end{prop}

Let $g$ and $g'$ be two elements of $G$ and consider the set $\mathcal{I}$ which indexes the conjugacy classes of $G.$ We suppose that $g\in C_\lambda$ and $g'\in C_\delta$ where $\lambda$ and $\delta$ are two elements of $\mathcal{I}.$ Since the characters of $G$ are constant on the conjugacy classes of $G$, the definition of $\mathfrak{F}$ can be extended to the conjugacy classes by the following way :
\begin{equation}\label{eq:var_ale}
\mathfrak{F}_{ {\bf{C}_\lambda} }(X):=|C_\lambda|\mathfrak{F}_g(X),
\end{equation}
for every $X\in \hat{G}.$ In addition, for every $X\in \hat{G}$ we have :
\begin{eqnarray*}
\mathfrak{F}_g(X)\mathfrak{F}_{g'}(X)&=&\frac{1}{|C_\lambda|}\mathfrak{F}_{ {\bf{C}_\lambda} }(X)\frac{1}{|C_\delta|}\mathfrak{F}_{ {\bf{C}_\delta} }(X)\\
&=&\frac{1}{|C_\lambda||C_\delta|}\frac{\mathcal{X}({\bf{C}_\lambda})}{\dim X}\frac{\mathcal{X}({\bf{C}_\delta})}{\dim X}\\
&=&\frac{1}{|C_\lambda||C_\delta|}\frac{\mathcal{X}({\bf{C}_\lambda}{\bf{C}_\delta})}{\dim X} ~~~~~~~~(\text{by Proposition \ref{caract_norma}})\\
&=&\sum_{\rho\in \mathcal{I}}c_{\lambda\delta}^\rho\frac{1}{|C_\lambda||C_\delta|}\frac{\mathcal{X}({\bf{C}_\rho})}{\dim X}\\
&=&\sum_{\rho\in \mathcal{I}}c_{\lambda\delta}^\rho\frac{|C_\rho|}{|C_\lambda||C_\delta|}\mathfrak{F}_{g_\rho}(X)~~~~~~~~(\text{by Equation \eqref{eq:var_ale}})
\end{eqnarray*} 
where $g_\rho$ is an arbitrary element of $C_\rho$ for every $\rho\in \mathcal{I}.$ 

\begin{prop}\label{prop:esp_ordre_2}
Let $G$ be a finite group and let $g\in C_\lambda$ and $g'\in C_\delta$ where $\lambda$ and $\delta$ are two elements of the finite set $\mathcal{I}$ which indexes the conjugacy classes of $G.$ We have :
$$\mathbb{E}_{\mathrm{P}_{G}}(\mathfrak{F}_{g}\mathfrak{F}_{g'})=\frac{c_{\lambda\delta}^{\emptyset}}{|C_\lambda||C_\delta|},$$
where $c_{\lambda\delta}^{\emptyset}$ is the (trivial) structure coefficient of ${\bf{C}}_{1_G}$ in the expansion of the product ${\bf{C}}_\lambda {\bf{C}}_\delta.$
\end{prop}
\begin{proof}
We have :
$$\mathfrak{F}_{g}\mathfrak{F}_{g'}=\sum_{\rho\in \mathcal{I}}c_{\lambda\delta}^\rho\frac{|C_\rho|}{|C_\lambda||C_\delta|}\mathfrak{F}_{g_\rho},$$
where $g_\rho$ is an arbitrary element of $C_\rho$ for every $\rho\in \mathcal{I}.$ Taking the expectations, we get :
$$\mathbb{E}_{\mathrm{P}_{G}}(\mathfrak{F}_{g}\mathfrak{F}_{g'})=\sum_{\rho\in \mathcal{I}}c_{\lambda\delta}^\rho\frac{|C_\rho|}{|C_\lambda||C_\delta|}\mathbb{E}_{\mathrm{P}_{G}}(\mathfrak{F}_{g_\rho}).$$ By Proposition \ref{esprerance}, we have $\mathbb{E}_{\mathrm{P}_{G}}(\mathfrak{F}_{g_\rho})=\delta_{\lbrace 1_G\rbrace}(g_\rho),$ thus we get the result.
\end{proof}

Next, we will show how the structure coefficients of the center of the group $G$ algebra intervene in the computation of $\mathbb{E}_{\mathrm{P}_G}(\mathfrak{F}^m_{g})$ for $m\geq 2.$

We have already showed that for every $X\in \hat{G}$ we have :
\begin{equation*}
\mathfrak{F}_g(X)\mathfrak{F}_{g'}(X)=\sum_{\rho\in \mathcal{I}}c_{\lambda\delta}^\rho\frac{|C_\rho|}{|C_\lambda||C_\delta|}\mathfrak{F}_{g_\rho}(X),
\end{equation*}
where $g_\rho$ is an arbitrary element of $C_\rho$ for any $\rho\in \mathcal{I}.$ This implies that for every $g\in C_\lambda$ and for any $X\in \hat{G},$ we have :
\begin{equation}\label{eq:moment_ordre_2}
\mathfrak{F}^2_g(X)=\sum_{\rho\in \mathcal{I}}c_{\lambda\lambda}^\rho\frac{|C_\rho|}{|C_\lambda||C_\lambda|}\mathfrak{F}_{g_\rho}(X),
\end{equation}
where $g_\rho$ is an arbitrary element of $C_\rho$ for any $\rho\in \mathcal{I}.$ If we multiply Equation \eqref{eq:moment_ordre_2} by $\mathfrak{F}_g(X),$ we get :
\begin{equation}
\mathfrak{F}^3_g(X)=\sum_{\rho\in \mathcal{I}}\sum_{\rho'\in \mathcal{I}}c_{\lambda\lambda}^\rho\frac{|C_\rho|}{|C_\lambda||C_\lambda|}c_{\rho\lambda}^{\rho'}\frac{|C_{\rho'}|}{|C_\rho||C_\lambda|}\mathfrak{F}_{g_{\rho'}}(X),
\end{equation}
where $g_{\rho'}$ is an arbitrary element of $C_{\rho'}$ for any $\rho'\in \mathcal{I}.$ Using Proposition \ref{esprerance}, we get :
$$\mathbb{E}_{\mathrm{P}_G}(\mathfrak{F}^3_{g})=\sum_{\rho\in \mathcal{I}}c_{\lambda\lambda}^\rho\frac{|C_\rho|}{|C_\lambda||C_\lambda|}c_{\rho\lambda}^{\emptyset}\frac{1}{|C_\rho||C_\lambda|}=c_{\lambda\lambda}^\lambda c_{\lambda\lambda}^{\emptyset}\frac{1}{|C_\lambda|^3}=\frac{c_{\lambda\lambda}^\lambda}{|C_\lambda|^2},$$
for any $g\in C_\lambda.$ Likewise, we can prove that :
\begin{equation*}
\mathbb{E}_{\mathrm{P}_G}(\mathfrak{F}^4_{g})=\sum_{\rho\in \mathcal{I}}\frac{c_{\lambda\lambda}^\rho c_{\rho\lambda}^\lambda}{|C_\lambda|^3},
\end{equation*}
for any $g\in C_\lambda.$ This highlights the fact that the mth-moments of the random variable $\mathfrak{F}_{g}$ for $m\geq 3$ can be expressed in terms of the structure coefficients of the center of the group $G$ algebra.

The reader should remark the particular structure coefficients that appear while computing the mth-moments of the random variable $\mathfrak{F}_{g}.$ These coefficients are:
$$c_{\lambda\lambda}^\emptyset,\,\, c_{\lambda\delta}^\emptyset,\,\, c_{\lambda\lambda}^\lambda,\,\, c_{\lambda\lambda}^\delta\text{ et } c_{\lambda\delta}^\delta.$$  
It will be then interesting to see if there exists explicit formulas for these coefficients in particular cases. In the most studied case where $G$ is the symmetric group $\mathcal{S}_n$ the set $\mathcal{I}$ is the set of partitions of size $n.$ If $\lambda$ and $\delta$ are two partitions of $n,$ there exists explicit formulas for the coefficients $c_{\lambda\emptyset}^{\delta}$ et $c_{(2,1^{n-2})\lambda}^\delta,$ see \cite{touPhd14}. Since every permutation of $\mathcal{S}_n$ is conjugated with itself, it is not difficult to prove that $c_{\lambda\delta}^\emptyset$ is the size of the conjugacy class of $\mathcal{S}_n$ associated to $\lambda$ if $\lambda=\delta$ and zero otherwise. 

We don't know any explicit formula for the other particular structure coefficients in the case of $\mathcal{S}_n.$ The structure coefficients $c_{\lambda\delta}^\rho$ of the center of the symmetric group algebra are (by multiplication by the size of the appropriate conjugacy class) symmetric on the partitions $\lambda$, $\delta$ and $\rho.$ This is an immediate consequence of the Frobenius theorem. Thus, an explicit formula for $c_{\lambda\lambda}^\delta$ implies an explicit formula for $c_{\lambda\delta}^\delta$ and vice-versa. It will be then interesting to try to find explicit formulas for the two following structure coefficients of the center of the symmetric group algebra:
$$c_{\lambda\lambda}^\lambda\text{ and } c_{\lambda\lambda}^\delta,$$  
where $\lambda$ and $\delta$ are two partitions of $n.$
It appears that an explicit formula exists for $c_{\lambda\lambda}^\delta$
when $\lambda$ is of the form $(2^k,1)$ \cite{Goupil-Personal}.

\section{Gelfand pairs and zonal spherical functions}\label{sec:paire_des_Gelfand}  

In \cite[chapitre VII]{McDo}, the reader can find a detailed introduction to the general theory of Gelfand pairs and zonal spherical functions. Here, we present this theory as a useful tool (similar to that of irreducible representations in the case of centers of finite group algebras) to describe the structure coefficients of double-class algebras. 

Throughout this section, we give generalisation of some properties of irreducible characters, presented in Section \ref{sec:irr_cha_fr_for}, to zonal spherical functions of Gelfand pairs. The most important among them is a Frobenius formula for the structure coefficients of the double-class algebras of Gelfand pairs which writes the structure coefficients in terms of zonal spherical functions.  

\subsection{Definitions}

We start this section by giving equivalent definitions for Gelfand pairs. Let $(G,K)$ be a pair where $G$ is a finite group and $K$ is a sub-group of $G.$ We consider the set $K\setminus G/ K$ of double-classes of $K$ in $G.$ We denote by $C(G,K)$ the set of functions $f:G\rightarrow \mathbb{C}$ which are constant on the double-classes of $K$ in $G.$ Explicitly, 
$$C(G,K):=\lbrace f:G\longrightarrow \mathbb{C} \text{ such that }f(kxk')=f(x) \text{ for every $x\in G$ and every $k,k'\in K$} \rbrace.$$
The set $C(G,K)$ is an algebra with multiplication defined by the convolution product of functions :
$$(fg)(x)=\sum_{y\in G}f(y)g(y^{-1}x)\text{ for any $f,g\in C(G,K)$}.$$

The algebra $C(G,K)$ is isomorphic to the double-class algebra $\mathbb{C}[K\setminus G/K].$ 
\begin{prop}\label{prop:isom}
The function $\psi$ defined as follows :
$$\begin{array}{ccccc}
\psi & : &\mathbb{C}[K\setminus G/ K] & \to & C(G,K) \\
& & x=\sum_{\lambda\in \mathcal{J}}x_\lambda \bf{DC}_\lambda & \mapsto &  \psi_x:=\sum_{\lambda\in \mathcal{J}}x_\lambda \delta_{DC_\lambda},\\
\end{array}$$
is an algebraic isomorphism.
\end{prop}
\begin{definition}
A pair $(G,K)$ where $G$ is a finite group and $K$ is a sub-group of $G$ is a Gelfand pair if its associated algebra $C(G,K)$ (or equivalently $\mathbb{C}[K\setminus G/ K]$ using Proposition \ref{prop:isom}) is commutative.
\end{definition}


There is another equivalent definition for Gelfand pairs which uses the representation theory. 
We consider the $G$-module $\mathbb{C}[G/K]$ where the action of $G$ on the set of left classes is defined by:
$$g(k_iK)=(gk_i)K \text{ for every $g\in G$ and every representative $k_i$ of $G/K$}.$$

If $V$ is a $G$-module, it can also be seen, in a natural way, as a $K$-module using the restriction action of $G$ on $K.$ This $K$-module is known as $\Res_K^G V.$ In a more complicated way, if $V$ is a $K$-module, we can always build a $G$-module denoted $\Ind_K^GV$ using $V.$ The matrix definition of this $G$-module is as follows :
$$\Ind_K^GV(g)=(V(k_i^{-1}gk_j))_{i,j},$$
where the $k_i$ are the representatives of the set $G/K=\lbrace k_1K,k_2K,\cdots ,k_lK\rbrace.$ 
\begin{theoreme}[Frobenius]\label{frob}
Let $G$ be a finite group and let $K$ be a sub-group of $G.$ If $X$ is a $G$-module and $Y$ is a $K$-module then we have :
$$<\Ind_K^G \mathcal{Y}, \mathcal{X}>=<\mathcal{Y},\Res_K^G \mathcal{X}>$$
where $\Res_K^G \mathcal{X}$ is the character function of the $K$-module $\Res_K^G X$ and $\Ind_K^G \mathcal{Y}$ is the character function of the $G$-module $\Ind_K^G Y.$
\end{theoreme}
\begin{proof}
See \cite[Theorem 1.12.6]{sagan2001symmetric} for a complete proof of this theorem.
\end{proof}
In the case $V={1}$ is the trivial $K$-module where the action of $K$ is defined by $g\cdot 1=1,$ $\Ind_K^GV$ can be identified with the $G$-module $\mathbb{C}[G/K].$  
\begin{prop}
The $G$-module $\mathbb{C}[G/K]$ can be decomposed into irreducible $G$-modules in such way that each irreducible $G$-module appears at most once in its decomposition if and only if the pair $(G,K)$ is a Gelfand pair. 
\end{prop}
\begin{proof}
See the result (1.1) in \cite[page 389]{McDo}.
\end{proof}
Suppose that $(G,K)$ is a Gelfand pair and that :
$$\mathbb{C}[G/K]=\bigoplus_{i=1}^{s}X_i,$$
where $X_i$ are irreducible $G$-modules. We define the functions $\omega_i:G\rightarrow \mathbb{C},$ using the irreducible characters $\mathcal{X}_i,$ as follows :
\begin{equation}
\omega_i(x)=\frac{1}{|K|}\sum_{k\in K}\mathcal{X}_i(x^{-1}k) \text{ for every $x\in G.$} 
\end{equation} 
The functions $(\omega_i)_{1\leq i\leq s}$ are called the \textit{zonal spherical functions} of the pair $(G,K).$ They have a long list of important properties, see \cite[page 389]{McDo}. They form an orthogonal basis for $C(G,K)$ and they satisfy the following equality :
\begin{equation}\label{prop_fon_zon}
\omega_i(x)\omega_i(y)=\frac{1}{|K|}\sum_{k\in K}\omega_i(xky),
\end{equation}
for every $1\leq i\leq s$ and for every $x,y\in G.$ 

\subsection{An example: the Gelfand pair $(G\times G^{opp},\diag(G))$}\label{sec:GelfandpairGxGopp}

We use $G^{opp}$ to denote the opposite group of $G$ and $\diag(G):=\lbrace(x,x^{-1}) \mid x\in G\rbrace$ to denote the diagonal sub-group of $G\times G^{opp}.$ 

In this section we show that for a finite group $G,$ the pair $(G\times G^{opp},\diag(G))$ is a Gelfand pair and its zonal sphercal functions are in fact the normalised irreducible characters of $G.$ This is explained in the introduction of Strahov in \cite{strahov2007generalized} and in the Example 9 in \cite[Section VII.1]{McDo}.

We first recall that if we have two groups $G$ and $H$ and if $X$ is a $G$-module and $Y$ is a $H$-module we can build a $G\times H$-module denoted $X\otimes Y.$ The matrix definition of this $G\times H$-module is obtained using the tensor product of matrices. If $A$ and $B$ are two matrices, $A\otimes B$ is the matrix obtained by multiplying all the entries $a_{ij}$ of the matrix $A$ by the matrix $B,$

$$A\otimes B=\begin{pmatrix}
a_{11}B & a_{12}B & \cdots \\
a_{21}B & a_{22}B & \cdots \\
\vdots & \vdots & \ddots
\end{pmatrix}.$$

We denote by $\mathcal{X}\otimes \mathcal{Y}$ the character of the $G\times H$-module $X\otimes Y.$ By definition of the tensor product of matrices, we have :
$$\tr(A\otimes B)=\sum_ia_{ii}\tr(B)=\tr(A)\tr(B),$$
for any two matrices $A$ and $B.$ Thus, for every $(g,h)\in G\times H,$ we have :
$$\mathcal{X}\otimes \mathcal{Y}(g,h)=\mathcal{X}(g)\mathcal{Y}(h).$$

If we have all the irreducible $G$-modules and all the irreducible $H$-modules, we can give all the irreducible $G\times H$-modules using the tensor product of modules. From \cite[Theorem 1.11.3]{sagan2001symmetric}, one can obtain the following :

$$\widehat{G\times H}=\hat{G}\otimes \hat{H},$$
where $$\hat{G}\otimes \hat{H}:= \lbrace X\otimes Y \text{ where $X\in \hat{G}$ and $Y\in \hat{H}$}\rbrace.$$ 

If $X$ is a $G$-module, $X$ can also be considered as a $G^{opp}$-module for the action $\cdot$ defined as follows :
$$g\cdot x=g^{-1}x,$$
for every $g\in G$ and every $x\in X.$ The character of $X$ seen as a $G^{opp}$-module is then the function $\overline{\mathcal{X}}.$ Suppose that
$$\hat{G}=\lbrace X_i \text{ such that } 1\leq i \leq |\mathcal{C}(G)|\rbrace,$$
then we have : 
$$\widehat{G\times G^{opp}}=\lbrace X_i\otimes \overline{X_j}\text{ such that } 1\leq i,j \leq |\mathcal{C}(G)|\rbrace.$$
\begin{prop}
If $G$ is a finite group, the pair $(G\times G^{opp},\diag(G))$ is a Gelfand pair.
\end{prop}
\begin{proof}
The irreducible $G\times G^{opp}$-modules have the form $X_i\otimes \overline{X_j}$ where $1\leq i,j \leq |\mathcal{C}(G)|.$ We have just to show that $X_i\otimes \overline{X_j}$ appears at most once in the decomposition of the $G\times G^{opp}$-module $\Ind_{\diag(G)}^{G\times G^{opp}}1.$ The number of appearances of the irreducible $G\times G^{opp}$-module $X_i\otimes \overline{X_j}$ in the decomposition of the $G\times G^{opp}$-module $\Ind_{\diag(G)}^{G\times G^{opp}}1$ is equal to :
$$<\Ind_{\diag(G)}^{G\times G^{opp}}1, \mathcal{X}_i\otimes \mathcal{\overline{X}}_j>,$$
where $\mathcal{X}_i\otimes \mathcal{\overline{X}}_j(x,y)=\mathcal{X}_i(x)\mathcal{\overline{X}}_j(y)$ for any $x,y\in G.$
By Frobenius reciprocity theorem (see Theorem \ref{frob}), we have : 
\begin{eqnarray*}
<\Ind_{\diag(G)}^{G\times G^{opp}}1, \mathcal{X}_i\otimes \mathcal{\overline{X}}_j>&=&<1, \Res_{\diag(G)}^{G\times G^{opp}} \mathcal{X}_i\otimes \mathcal{\overline{X}}_j>\\
&=&\frac{1}{|\diag(G)|}\sum_{(g,g^{-1})\in \diag(G)}\mathcal{X}_i\otimes \mathcal{\overline{X}}_j(g,g^{-1})\,\,\,\,\,\,\text{(by definition of $<,>$)}\\
&=&\frac{1}{|G|}\sum_{g\in G}\mathcal{X}_i(g)\mathcal{\overline{X}}_j(g^{-1})\\
&=&<\mathcal{X}_i,\mathcal{\overline{X}}_j>\\
&=&\delta_{\mathcal{X}_i,\mathcal{\overline{X}}_j}.
\end{eqnarray*}
Therefore, we have :
$$\Ind_{\diag(G)}^{G\times G^{opp}}1=\sum_{i=1}^{|\mathcal{C}(G)|}X_i\otimes X_i.$$
\end{proof}
We deduce from this result the following corollary.
\begin{cor}\label{cor:nbclass=nbdoublclass} Let $G$ be a finite group. Then, we have :
$$|\mathcal{C}(G)|=|\diag(G)\setminus G\times G^{opp}/\diag(G)|.$$
\end{cor}
The zonal spherical functions $(\omega_i)_{1\leq i \leq |\mathcal{C}(G)|}$ of the pair $(G\times G^{opp},\diag(G))$ are defined as follows :
\begin{eqnarray}\label{fct_sph_zonal=chara}
\omega_i(x,y)&=&\frac{1}{|G|}\sum_{g\in G}\mathcal{X}_i\otimes \mathcal{X}_i((x,y)^{-1}(g,g^{-1})) \nonumber \\
&=&\frac{1}{|G|}\sum_{g\in G}\mathcal{X}_i(x^{-1}g) \mathcal{X}_i(g^{-1}y^{-1}) \nonumber \\
&=&\frac{1}{\dim X_i}\mathcal{X}_i(yx),
\end{eqnarray}
the last equality comes from the definition of the product and from Proposition \ref{caract_norma}. By using the property \eqref{prop_fon_zon} of zonal spherical function in this case, we get :
\begin{equation*}
\omega_i(x,1)\omega_i(y,1)=\frac{1}{|G|}\sum_{g\in G}\omega_i((x,1)(g,g^{-1})(y,1)),
\end{equation*}
for any $x,y\in G.$ If we develop this equality, we obtain :
\begin{equation*}
\frac{1}{\dim X_i}\mathcal{X}_i(x)\frac{1}{\dim X_i}\mathcal{X}_i(y)=\frac{1}{|G|}\sum_{g\in G}\frac{1}{\dim X_i}\mathcal{X}_i(g^{-1}xgy),
\end{equation*}
for any $x,y\in G.$ If we extend this relation by linearity to the group algebra $\mathbb{C}[G],$ and if $x$ and $y$ are two elements of $Z(\mathbb{C}[G]),$ we get :
\begin{equation*}
\frac{1}{\dim X_i}\mathcal{X}_i(x)\frac{1}{\dim X_i}\mathcal{X}_i(y)=\frac{1}{\dim X_i}\mathcal{X}_i(xy).
\end{equation*}
Thus we re-find the result of Proposition \ref{caract_norma}. This allows us to look to the zonal spherical functions as generalisations of normalised irreducible characters. In particular, the formula \eqref{prop_fon_zon} generalises the result of Proposition \ref{caract_norma}.

\section{Generalisation of some properties of irreducible characters to zonal spherical functions}\label{sec:Gen_fct_sph_zon}

In this section we give some results concerning zonal spherical functions of Gelfand pairs which generalise the properties of irreducible characters of finite groups given in Section \ref{sec:irr_cha_fr_for}. As similar to Proposition \ref{caract_norma}, we show that the zonal spherical functions of a given Gelfand pair are morphisms. Our main result is a Frobenius formula in the case of Gelfand pairs, this formula allows us to write the structure coefficients of the double-class algebra associated to a Gelfand pair in terms of zonal spherical functions. 

The properties concerning zonal spherical functions of Gelfand pair in general, given in this section, appear -- according to the author knowledge -- for the first time here. Some authors have already given these properties in special cases of Gelfand pairs, the reader can see the paper \cite{jackson2012character} of Jackson and Sloss about the Gelfand pair $(\mathcal{S}_n\times \mathcal{S}_{n-1}^{opp},\diag(\mathcal{S}_{n-1}))$ and that \cite{GouldenJacksonLocallyOrientedMaps} of Goulden and Jackson about the Gelfand pair $(\mathcal{S}_{2n},\mathcal{B}_n).$ 

\begin{prop}\label{homomorphism}
The zonal spherical functions of a given Gelfand pair $(G,K)$ define morphisms between the algebra $\mathbb{C}[K\setminus G/ K]$ and $\mathbb{C}^*.$
\end{prop}
\begin{proof}
Since they are defined on the group $G,$ the zonal spherical functions of the Gelfand pair $(G,K)$  can be linearly extended to be defined on the group algebra $\mathbb{C}[G].$ In this case, these functions preserve the property given by equation \eqref{prop_fon_zon}. If $x$ and $y$ are two elements of $\mathbb{C}[K\setminus G/ K],$ then we have :
\begin{equation*}
\frac{1}{|K|}\sum_{k\in K}\omega_i(xky)=\frac{1}{|K|}\sum_{k\in K}\omega_i(xy)=\omega_i(xy).
\end{equation*}
Thus, we obtain : $\omega_i(xy)=\omega_i(x)\omega_i(y)$ for every $x,y \in \mathbb{C}[K\setminus G/ K],$ which ends the proof.
\end{proof}

We recall the finite set $\mathcal{J}$ whose elements $\lambda$ index the double-classes of $K$ in $G.$ For every $\lambda\in \mathcal{J},$ we associate the function $f_\lambda: G\rightarrow \mathbb{C}$ defined as follows :
$$f_\lambda(g)=\left\{
\begin{array}{ll}
  1 & \qquad \mathrm{if}\quad g\in DC_\lambda, \\
  0 & \qquad \mathrm{ifnot.}\quad \\
 \end{array}
 \right.$$
It isn't difficult to verify that the family $(f_\lambda)_{\lambda\in \mathcal{J}}$ form a basis for $C(G,K),$ therefore :
\begin{equation}\label{dimC(G,K)}
\dim C(G,K) = |K\setminus G/K|.
\end{equation}
\begin{prop}\label{egalité_doubleclasse_fctzonale}
If $(G,K)$ is a Gelfand pair then the number of its zonal spherical functions is equal to the number of double-classes of $K$ in $G.$
\end{prop}
\begin{proof}
Using equation \eqref{dimC(G,K)}, the dimension of $C(G,K)$ is equal to the number of double-classes of $K$ in $G.$ Moreover, since $(G,K)$ is a Gelfand pair, its zonal spherical functions form a basis for $C(G,K),$ thus we get the result.
\end{proof}

By Proposition \ref{egalité_doubleclasse_fctzonale}, if $(G,K)$ is a Gelfand pair then we have $|\mathcal{J}|$ zonal spherical functions for $(G,K).$ We will use a set which we denote $\mathcal{J}^{'}$ ($|\mathcal{J}^{'}|=|\mathcal{J}|$) to index the zonal spherical functions of $(G,K).$ If $(G,K)$ is a Gelfand pair, then we can write :
$$\mathbb{C}[G/K]=\bigoplus_{\lambda\in \mathcal{J}^{'}}X^\lambda,$$
where the $X^\lambda$ are the irreducible $G$-modules. Let $\theta\in \mathcal{J}^{'},$ since zonal spherical functions of $(G,K)$ are constant on the double-classes of $K$ in $G,$ if $\lambda\in \mathcal{J},$ we denote by $\omega^{\theta}_\lambda$ the value of the zonal spherical function $\omega^{\theta}$ on any element of the double-class $DC_\lambda,$
$$\omega^{\theta}_\lambda:=\omega^{\theta}(g),$$
for an element $g\in DC_\lambda.$ Among the remarkable properties of zonal spherical functions given in \cite[chapitre VII]{McDo}, we recall the following :
\begin{enumerate}\label{1}
\item[1-] $\omega^{\theta}\omega^{\psi}=\delta_{\theta\psi}\frac{|G|}{\mathcal{X}^{\theta}(1)}\omega^{\theta}$
\item[2-]$<\omega^{\theta},\omega^{\psi}>=\delta_{\theta\psi}\frac{1}{\mathcal{X}^{\theta}(1)}$
\end{enumerate}
where $\mathcal{X}^{\theta}(1)=\dim X^\theta.$

\begin{notation}
If $(G,K)$ is a Gelfand pair, we denote by $\widehat{(G,K)}$ the set of its zonal spherical functions. 
\end{notation}

\begin{prop}\label{isom_algebredoubleclasse_fct}
Let $(G,K)$ be a Gelfand pair, then the function 
$$\begin{array}{ccccc}
\mathbb{C}[K\setminus G / K] & \to & \mathcal{F}(\widehat{(G,K)},\mathbb{C}) \\
x & \mapsto & (\omega^{\lambda}\mapsto \omega^{\lambda}(x)) \\
\end{array}$$
is an algebras isomorphism where $\mathcal{F}(\widehat{(G,K)},\mathbb{C})$ is the set of functions defined on $\widehat{(G,K)}$ with values in $\mathbb{C}.$
\end{prop}
\begin{proof}
This proposition is a direct consequence from Propositions \ref{homomorphism} and \ref{egalité_doubleclasse_fctzonale}.
\end{proof}

Proposition \ref{isom_algebredoubleclasse_fct} is the analogous of Proposition \ref{isom_centre_fct} for zonal spherical functions of Gelfand pairs. For a given element $\theta\in \mathcal{J}^{'},$ we define the element $E_\theta$ as follows :
$$E_\theta:=\frac{\mathcal{X}^\theta(1)}{|G|}\sum_{g\in G}\overline{\omega^{\theta}(g)}g=\frac{\mathcal{X}^\theta(1)}{|G|}\sum_{\rho\in \mathcal{J}}\overline{\omega^{\theta}_\rho}\bf{DC}_\rho.$$
\begin{prop}\label{eq:dc_foncsphe}
The family $(E_\theta)_{\theta\in \mathcal{J}^{'}}$ is a basis for $\mathbb{C}[K\setminus G/K]$ formed by idempotent elements. Moreover, we have :
$${\bf{DC}_\lambda}=|DC_\lambda|\sum_{\psi\in \mathcal{J}^{'}}\omega^\psi_\lambda E_\psi.$$
\end{prop}
\begin{proof}
The fact that $(E_\theta)_{\theta\in \mathcal{J}^{'}}$ is a family of idempotent elements comes from the second property above concerning zonal spherical functions. In fact, the image of $E_\theta$ by the isomorphism given in Proposition \ref{isom_algebredoubleclasse_fct} is the following function
$$\begin{array}{ccccc}
& & \widehat{(G,K)} & \to & \mathbb{C} \\
& & \omega^\phi & \mapsto & \left\{
\begin{array}{ll}
  1 & \qquad \mathrm{si}\quad \phi=\theta \\
  0 & \qquad \mathrm{sinon}\quad \\
 \end{array}
 \right. \\
\end{array}.$$
because $$\omega^\phi(E_\theta)=\mathcal{X}^\theta(1)<\omega^\phi,\omega^\theta>,$$
for every $\phi\in \mathcal{J}^{'}.$ To prove that ${\bf{DC}_\lambda}=|DC_\lambda|\sum_{\psi\in \mathcal{J}^{'}}\omega^\psi_\lambda E_\psi,$ we develop the sum on the right hand :
\begin{eqnarray*}
|DC_\lambda|\sum_{\psi\in \mathcal{J}^{'}}\omega^\psi_\lambda E_\psi &=&\sum_{\psi\in \mathcal{J}^{'}}\omega^\psi_\lambda  \frac{\mathcal{X}^\psi(1)}{|G|}\sum_{\delta\in \mathcal{J}}|DC_\lambda|\overline{\omega^{\psi}_\delta}\bf{DC}_\delta\\
&=&\sum_{\psi\in \mathcal{J}^{'}}\sum_{\delta\in \mathcal{J}}\mathcal{X}^\psi(1)\frac{|DC_\lambda|}{|G|}\omega^\psi_\lambda  \overline{\omega^{\psi}_\delta}\bf{DC}_\delta\\
&=&\sum_{\psi\in \mathcal{J}^{'}}\mathcal{X}^\psi(1)\frac{|DC_\lambda|}{|G|}\omega^\psi_\lambda \overline{\omega^{\psi}_\lambda}\bf{DC}_\lambda\\
&=&{\bf{DC}_\lambda} \,\,\,\, (\text{because $\sum_{\psi\in \mathcal{J}^{'}}\frac{|DC_\lambda|}{|G|}\omega^\psi_\lambda  \overline{\omega^{\psi}_\lambda}=<\omega^\psi,\omega^\psi>=\frac{1}{\mathcal{X}^\psi(1)}$ }).
\end{eqnarray*}
It is clear that the elements $(E_\theta)_{\theta\in \mathcal{J}^{'}}$ form a basis for $\mathbb{C}[K\setminus G/K]$ since $|\mathcal{J}^{'}|$ is equal to the number of double-classes of $K$ in $G.$
\end{proof}

This proposition allows us to give a formula for the structure coefficients of the double-class algebra $\mathbb{C}[K\setminus G/K]$ of a Gelfand pair $(G,K)$ in terms of its zonal spherical functions.
\begin{theoreme}\label{Th_tt}
Let $(G,K)$ be a Gelfand pair and let $\mathcal{J}$ be the set which indexes the double-classes of $K$ in $G.$ Let $\lambda,\delta$ and $\rho$ be three elements of $\mathcal{J},$ the structure coefficient $k_{\lambda\delta}^\rho$ of the double-class algebra $\mathbb{C}[K\setminus G/K]$ can be written in terms of zonal spherical functions as follows:
$$k_{\lambda\delta}^\rho=\frac{|DC_\lambda||DC_\delta|}{|G|}\sum_{\theta\in \mathcal{J}^{'}}\mathcal{X}^\theta(1)\omega^\theta_\lambda \omega^\theta_\delta \overline{\omega^{\theta}_\rho}.$$
\end{theoreme}
\begin{proof}
By Proposition \ref{eq:dc_foncsphe}, we can write :
\begin{eqnarray*}
\bf{DC}_\lambda\bf{DC}_\delta &=& |DC_\lambda|\sum_{\theta\in \mathcal{J}^{'}}\omega^\theta_\lambda E_\theta |DC_\delta|\sum_{\phi\in \mathcal{J}^{'}}\omega^\phi_\delta E_\phi \\
&=&|DC_\lambda||DC_\delta|\sum_{\theta\in \mathcal{J}^{'}}\omega^\theta_\lambda \omega^\theta_\delta E_\theta \\
&=&|DC_\lambda||DC_\delta|\sum_{\theta\in \mathcal{J}^{'}}\omega^\theta_\lambda \omega^\theta_\delta \frac{\mathcal{X}^\theta(1)}{|G|}\sum_{\rho\in \mathcal{J}}\overline{\omega^{\theta}_\rho}\bf{DC}_\rho,
\end{eqnarray*}
and the result follows.
\end{proof}

Theorem \ref{Th_tt} is an analogous of (Frobenius) Theorem \ref{coef_fct_cara}. The author did not find such theorem in the literature which treats the general case of Gelfand pairs. A similar theorem was given for particular cases of Gelfand pairs, see for example the paper \cite{jackson2012character} of Jackson and Sloss about the Gelfand pair $(\mathcal{S}_n\times \mathcal{S}_{n-1}^{opp},\diag(\mathcal{S}_{n-1}))$ and that \cite{GouldenJacksonLocallyOrientedMaps} of Goulden and Jackson about the Gelfand pair $(\mathcal{S}_{2n},\mathcal{B}_n).$

Theorem \ref{Th_tt} can be generalised to give an expression for the structure coefficients, of the product of more than two double-classes, in terms of zonal spherical functions in the case of Gelfand pairs. Especially, we have: 
\begin{theoreme}\label{Th_prin_gener}
Let $(G,K)$ be a Gelfand pair and let $\mathcal{J}$ be the set which indexes the double-classes of $K$ in $G.$ Let $\lambda_1,\lambda_2,\cdots,\lambda_r$ and $\rho$ be $k+1$ elements (we count with multiplicity which means that these are not necessarily distinct elements) of $\mathcal{J},$ the structure coefficient $k_{\lambda_1\lambda_2\cdots\lambda_r}^\rho$ of $\bf{DC}_\rho$ in the expansion of the product $\bf{DC}_{\lambda_1}\bf{DC}_{\lambda_2}\cdots\bf{DC}_{\lambda_r}$ in the double-class algebra $\mathbb{C}[K\setminus G/K]$ can be written in terms of zonal spherical functions as follows :
\begin{equation}\label{eq:coef_str_gener}
k_{\lambda_1\lambda_2\cdots\lambda_r}^\rho=\frac{|DC_{\lambda_1}||DC_{\lambda_2}|\cdots|DC_{\lambda_r}|}{|G|}\sum_{\theta\in \mathcal{J}^{'}}\mathcal{X}^\theta(1)\omega^\theta_{\lambda_1} \omega^\theta_{\lambda_2}\cdots \omega^\theta_{\lambda_r}\overline{\omega^{\theta}_\rho}.
\end{equation}
\end{theoreme}
\begin{proof}
We reason by induction on $r.$ The case $r=2$ is proved in Theorem \ref{Th_tt}. Suppose that the equation \eqref{eq:coef_str_gener} is true for $r-1$ double-classes and let us prove it for $r.$ If we write:
\begin{align*}
\bf{DC}_{\lambda_1}\bf{DC}_{\lambda_2}\cdots\bf{DC}_{\lambda_r}&=\sum_{\delta\in \mathcal{J}}k_{\lambda_1\lambda_2\cdots\lambda_{r-1}}^{\delta}\bf{DC}_{\delta}\bf{DC}_{\lambda_r}&\\
&=\sum_{\delta\in \mathcal{J}}\sum_{\rho\in \mathcal{J}}k_{\lambda_1\lambda_2\cdots\lambda_{r-1}}^{\delta}k_{\delta\lambda_r}^\rho\bf{DC}_{\rho},
\end{align*}
then one can see that :
\begin{equation*}
k_{\lambda_1\lambda_2\cdots\lambda_r}^\rho=\sum_{\delta\in \mathcal{J}}k_{\lambda_1\lambda_2\cdots\lambda_{r-1}}^{\delta}k_{\delta\lambda_r}^\rho.
\end{equation*}
By using the induction hypothesis and after simplification we get our result.
\end{proof}

\section{Applications and already known results.}\label{sec:app}

In this section we apply our main theorem, Theorem \ref{Th_tt}, to special cases of Gelfand pairs. Using the Gelfand pair $(G\times G^{opp},\diag(G)),$ we show that the Frobenius theorem is a special case of our main theorem.

\subsection{How to get the Frobenius theorem by using the Gelfand pair $(G\times G^{opp},\diag(G))$}\label{sec:FrobeniustheorembyusingtheGelfandpair}

We showed in Section \ref{sec:GelfandpairGxGopp} that the zonal spherical functions of the Gelfand pair $(G\times G^{opp},\diag(G))$ are not but the normalised irreducible characters of $G.$ Let us now go further in investigating this pair. For an element $(a,b)\in G\times G^{opp}$ and two elements $(x,x^{-1})$ and $(y,y^{-1})$ of $\diag(G)$, we have :
\begin{equation}\label{pr_double_classe}
(x,x^{-1})(a,b)(y,y^{-1})=(xay,y^{-1}bx^{-1}).
\end{equation}

This equality allows us to link, by a direct way, the double-classes of $\diag(G)$ in $G\times G^{opp}$ to the conjugacy classes of $G.$ In fact, the equality \eqref{pr_double_classe} implies that two elements $(a,b)$ and $(c,d)$ are in the same double-class of $\diag(G)$ in $G\times G^{opp}$ if and only if $ab$ and $cd$ are conjugated in $G.$

Recall the set $\mathcal{I}$ used to index the conjugacy classes of $G.$ The set $\mathcal{I}$ indexes also the double-classes of $\diag(G)$ in $G\times G^{opp}.$ Let $\lambda\in \mathcal{I},$ the double-class of $\diag(G)$ in $G\times G^{opp}$ associated to $\lambda$ is denoted $C'_\lambda.$ Explicitly, we have : 
$$C'_\lambda=\lbrace (a,b)\in G\times G^{opp}~|~ ab\in C_\lambda\rbrace.$$
What we have just written is a direct proof to Corollary \ref{cor:nbclass=nbdoublclass}. 
\begin{prop}\label{prop:relation_taille_dc_class}
If $\lambda$ is an element of $\mathcal{I}$ then we have :
$$|C'_\lambda|=|G||C_\lambda|.$$
\end{prop}
\begin{proof}
The function which sends an element $(x,y)$ of $G\times C_\lambda$ to $(x,x^{-1}y)\in G\times G^{opp}$ is a bijection between $G\times C_\lambda$ and $C'_\lambda.$ The function which sends an element $(a,b)$ of $C'_\lambda$ to $(a,ab)$ is its inverse function.
\end{proof}

Let $\lambda$ and $\delta$ be two elements of $\mathcal{I},$ the structure coefficients $c'^\rho_{\lambda\delta}$ of the double-class algebra of $\diag(G)$ in $G\times G^{opp}, \mathbb{C}[\diag(G)\setminus G\times G^{opp}/\diag(G)],$ are defined by the following equation :
\begin{equation}
\mathbf{C}'_\lambda\mathbf{C}'_\delta=\sum_{\rho\in \mathcal{I}}c'^\rho_{\lambda\delta}\mathbf{C}'_\rho.
\end{equation} 

\begin{prop}\label{prop:lien_taille_classe_et_double_classe}
Let $G$ be a finite group and let $\mathcal{I}$ be the set which indexes the conjugacy classes of $G.$ If $\lambda,\delta$ and $\rho$ are three elements of $\mathcal{I}$ then the structure coefficients $c_{\lambda\delta}^{\rho}$ and $c'^\rho_{\lambda\delta}$ of $Z(\mathbb{C}[G])$ and $\mathbb{C}[\diag(G)\setminus G\times G^{opp}/\diag(G)]$ are strongly related by :
\begin{equation*}
c'^\rho_{\lambda\delta}=|G|c^\rho_{\lambda\delta}.
\end{equation*}
\end{prop}
\begin{proof}
By direct computing, using Proposition \ref{desc_coef},the structure coefficient $c_{\lambda\delta}^\rho$ in $Z(\mathbb{C}[G])$ given by equation \eqref{coef_centre} is not but the cardinal of the set $A_{\lambda\delta}^{\rho}$ defined by :
$$A_{\lambda\delta}^{\rho}=\lbrace (x,y)\in C_\lambda\times C_{\delta}\text{ such that } xy=z \rbrace,$$
where $z$ is a fixed element of $C_\rho.$
By the same way, we have $c'^\rho_{\lambda\delta}=|A'^{\rho}_{\lambda\delta}|$ where
$$A'^{\rho}_{\lambda\delta}=\lbrace \big((x_1,x_2),(y_1,y_2)\big)\in C'_\lambda\times C'_\delta \text{ such that } (x_1y_1,y_2x_2)=(z,1)\rbrace.$$

The function, 
$$\begin{array}{ccccc}
&  & A'^{\rho}_{\lambda\delta} & \to & A^{\rho}_{\lambda\delta}\times G \\
& & \big((x_1,x_2),(y_1,x_2^{-1})\big) & \mapsto & \big((x_1x_2,x_2^{-1}y_1),x_2\big) \\
\end{array}$$
is well defined -- the image of an element of $A'^{\rho}_{\lambda\delta}$ is in $A^{\rho}_{\lambda\delta}\times G$ -- and it is a bijection with inverse defined by the function
\begin{equation*}
\big((x,y),t\big)\mapsto \big((xt^{-1},t),(ty,t^{-1})\big). \qedhere
\end{equation*}
\end{proof}

We have now all necessary informations to re-obtain the Frobenius theorem as an application of our Theorem \ref{Th_tt}. The reader should remark, using the results of this section and Section \ref{sec:GelfandpairGxGopp}, that the study of the center of a finite group $G$ algebra is equivalent to the study of the double-class algebra of the pair $(G\times G^{opp},\diag(G)).$

We showed in Section \ref{sec:GelfandpairGxGopp} that the number of zonal spherical functions of the pair $(G\times G^{opp},\diag(G))$ is equal to the number of irreducible characters of $G$ and that the zonal spherical functions of $(G\times G^{opp},\diag(G))$ are not but the normalised irreducible characters of $G.$ By using Proposition \ref{prop:relation_taille_dc_class}, Theorem \ref{Th_tt} when applied to the Gelfand pair $(G\times G^{opp},\diag(G))$ gives us the following :
\begin{equation*}
c'^\rho_{\lambda\delta}=\frac{|C_\lambda||G||C_\delta||G|}{|G|}\sum_{X\in \hat{G}}\mathcal{X}(1)\frac{\mathcal{X}_\lambda\mathcal{X}_\delta\overline{\mathcal{X}_\rho}}{\dim X\dim X\dim X}.
\end{equation*}
After simplification and by using Proposition \ref{prop:lien_taille_classe_et_double_classe}, we get :
\begin{equation*}
c^\rho_{\lambda\delta}=|C_\lambda||C_\delta|\sum_{X\in \hat{G}}\frac{\mathcal{X}_\lambda\mathcal{X}_\delta\overline{\mathcal{X}_\rho}}{\dim^2 X}.
\end{equation*}
This equation is another form of Theorem \ref{coef_fct_cara} of Frobenius which is often used when taking the graphs interpretation of the structure coefficients of the centers of group algebras rather than the algebraic one (which is used in this paper). The reader can see the appendix of Zagier in \cite{lando2004graphs} for more informations.

Here as well as in Theorem \ref{coef_fct_cara}, we study the structure coefficients that appear in the product of two basis elements. In the appendix \cite{lando2004graphs} of Zagier, the author was interested in the structure coefficients that appear in the product of $k$-basis elements, where $k$ is a finite integer. He gave thus the general form of the Frobenius formula. I would like to mention that using the same idea presented in this section, we can re-obtain his general form of the Frobenius formula as an application of our general theorem, Theorem \ref{Th_prin_gener}.

\subsection{The Gelfand pair $(\mathcal{S}_{2n},\mathcal{B}_n)$.} 

The pair $(\mathcal{S}_{2n},\mathcal{B}_n),$ where $\mathcal{S}_{2n}$ is the symmetric group of $[2n]$ and $\mathcal{B}_n$ is its hyperoctahedral sub-group, is a Gelfand pair. In \cite[Chapter VII, 2]{McDo}, the reader can find a detailed work about this pair. The double-classes of $\mathcal{B}_n$ in $\mathcal{S}_{2n}$ are indexed by partitions of $n.$ If $\lambda$ is a partition of $n,$ its associated $\mathcal{B}_n$-double-class in $\mathcal{S}_{2n},$ which we denote $K_\lambda,$ is the set of all permutations of $[2n]$ with coset-type $\lambda.$ The definition of the coset-type of a permutation of $2n$ can be found on Page 401 of \cite{McDo}. In short, if $p$ is a permutation of $2n,$ we consider all the cycles of the form
\begin{equation}\label{eq:cycleform_cycletype}
i\triangleright p(i) \longrightarrow \overline{p(i)}\triangleleft p^{-1}(\overline{p(i)})\Longrightarrow\overline{p^{-1}(\overline{p(i)})}\cdots \Longrightarrow i\end{equation}
of $p$ where $\overline{2k}$ is $2k-1$ and $\overline{2k-1}$ is $2k$ for any integer $1\leq k\leq n.$ These cycles are of even length's. The coset-type of $p$ is the partition of $n$ obtained using the halves of the lengths of these cycles. For example, take the following permutation $t$ of $12,$
$$t=\begin{pmatrix}
1&2&3&4&5&6&7&8&9&10&11&12\\
8&12&4&6&10&9&11&1&7&2&3&5
\end{pmatrix}.$$ The cycles of $t$ of form \eqref{eq:cycleform_cycletype} are:
$$1\triangleright 8 \longrightarrow 7 \triangleleft 9 \Longrightarrow 10 \triangleright 2 \longrightarrow 1 \triangleleft 8 \Longrightarrow 7 \triangleright 11 \longrightarrow 12 \triangleleft 2 \Longrightarrow 1,$$

$$3\triangleright 4 \longrightarrow 3 \triangleleft 11 \Longrightarrow 12 \triangleright 5 \longrightarrow 6 \triangleleft 4\Longrightarrow 3$$
and 
$$5\triangleright 10 \longrightarrow 9 \triangleleft 6 \Longrightarrow 5.$$
The lengths of these cycles are 6, 4 and 2 respectively, thus the coset-type of $t$ is the partition $(3,2,1)$ of $6.$

If $\lambda,$ $\delta$ and $\rho$ are three partitions of $n,$ the structure coefficient $b_{\lambda\delta}^\rho$ of the double-class algebra $\mathbb{C}[\mathcal{B}_{n}\setminus \mathcal{S}_{2n}/ \mathcal{B}_{n}]$ is the coefficient of ${\bf K}_\rho$ in the product ${\bf K}_\lambda {\bf K}_\delta.$

Since the double-classes of $\mathcal{B}_n$ in $\mathcal{S}_{2n}$ are indexed by partitions of $n,$ the zonal spherical functions of the Gelfand pair $(\mathcal{S}_{2n},\mathcal{B}_n)$ are also indexed by the set of partitions of $n$ by Proposition \ref{egalité_doubleclasse_fctzonale}. If $\rho$ is a partition of $n,$ its associated zonal spherical function $\omega^{\rho}$ is defined by : $$\omega^{\rho}(x)=\frac{1}{|\mathcal{B}_n|}\sum_{k\in \mathcal{B}_n}\chi^{2\rho}(xk),$$
for $x\in \mathcal{S}_{2n}$, where $\chi^{2\rho}$ is the character of the irreducible $\mathcal{S}_{2n}$-module corresponding to $2\rho:=(2\rho_1,2\rho_2,\cdots).$

Using Theorem \ref{Th_tt}, the coefficient $b_{\lambda\delta}^\rho$ is equal to :
\begin{equation}\label{str_coef_hecke_alg}
b_{\lambda\delta}^\rho = \frac{|K_\lambda||K_\delta|}{|\mathcal{S}_{2n}|}\sum_{\theta\in \mathcal{P}_n}\mathcal{X}^{2\theta}(1)\omega^\theta_\lambda \omega^\theta_\delta \overline{\omega^{\theta}_\rho}.
\end{equation}
If for two partitions $\theta$ and $\lambda$ of $n$ we define $\phi^\theta(\lambda)$ to be :
$$\phi^\theta(\lambda):=|K_\lambda|\omega^\theta_\lambda,$$
then from equation \eqref{str_coef_hecke_alg} one can re-obtain the following result given in \cite[Lemma 3.3]{hanlon1992some} by Hanlon, Stanley and Stembridge.
\begin{lemma} Let $\lambda,$ $\delta$ and $\rho$ be three partitions of $n$ then we have :
\begin{equation*}
b_{\lambda\delta}^\rho = \frac{1}{|K_\rho|}\sum_{\theta\in \mathcal{P}_n}\frac{1}{H_{2\theta}}\phi^\theta(\lambda)\phi^\theta(\delta)\phi^\theta(\rho),
\end{equation*}
where $H_{2\theta}=\frac{\mathcal{X}^{2\theta}(1)}{|\mathcal{S}_{2n}|}$ is the product of hook-lengths of $2\theta.$
\end{lemma}
This result can be found also in the paper \cite{GouldenJacksonLocallyOrientedMaps} of Goulden and Jackson.

\subsection{The Gelfand pair $(\mathcal{S}_n\times \mathcal{S}_{n-1}^{opp},\diag(\mathcal{S}_{n-1}))$} There are many ways to represent a permutation of $n.$ Here we consider the product of disjoint cycles way. If $\omega$ is a permutation of $n,$ 
$$\omega=\mathfrak{C}_1\mathfrak{C}_2\cdots\mathfrak{C}_k,$$then we define the \textit{cycle-type} of $\omega,$ denoted by cycle-type($\omega$), to be the partition of $n$ obtained using the lenghts of the disjoint cycles $\mathfrak{C}_i$ of $\omega.$ For example, the permutation $(167)(23)(4)(58)$ of $8$ has the cycle-type $(3,2,2,1).$

The pair $(\mathcal{S}_n\times \mathcal{S}_{n-1}^{opp},\diag(\mathcal{S}_{n-1}))$ was studied by Brender in 1976. In 2007, Strahov shows in \cite[Proposition 2.2.1]{strahov2007generalized} that it is a Gelfand pair. The set of $\diag(\mathcal{S}_{n-1})$-double-classes is indexed by the pairs $(i,\lambda)$ where $i$ is an integer between $1$ and $n$ and $\lambda$ is a partition of $n-i.$ The associated $\diag(\mathcal{S}_{n-1})$-double-class of such a pair $(i,\lambda)$ is as follows :
$$DC_{(i,\lambda)}=\lbrace (a,b)\in \mathcal{S}_n\times \mathcal{S}^{opp}_{n-1} \text{ such that } ab\in C_{(i,\lambda)} \rbrace,$$
where
$$C_{(i,\lambda)}=\lbrace x\in \mathcal{S}_n \text{ such that $1$ is in a cycle $c$ with length $i$ and cycle-type}(x\setminus c)=\lambda\rbrace.$$ 
The size of $DC_{(i,\lambda)}$ is known to be :
$$|DC_{(i,\lambda)}|=|S_{n-1}||C_{(i,\lambda)}|=\frac{(n-1)!^2}{z_\lambda}.$$

Let $i$ and $j$ be two integers between $1$ and $n$ and let $\lambda$ and $\delta$ be two partitions of $n-i$ and $n-j$ respectively. The structure coefficients $c_{(i,\lambda)(j,\delta)}^{(r,\rho)}$ of the double-class algebra $\mathbb{C}[\diag(\mathcal{S}_{n-1}))\setminus \mathcal{S}_n\times \mathcal{S}^{opp}_{n-1}/ \diag(\mathcal{S}_{n-1}))]$ are defined by the following equation :

\begin{equation*}{\bf DC}_{(i,\lambda)}{\bf DC}_{(j,\delta)}=\sum_{1\leq r\leq n \atop{\rho\vdash n-r}}c^{(r,\rho)}_{(i,\lambda)(j,\delta)}{\bf DC}_{(r,\rho)}.
\end{equation*}

According to Theorem \ref{Th_tt}, the coefficient $c^{(r,\rho)}_{(i,\lambda)(j,\delta)}$ is equal to :
\begin{eqnarray}\label{str_coef_doub_class_diag(S_n)}
c^{(r,\rho)}_{(i,\lambda)(j,\delta)} &=& \frac{|S_{n-1}||C_{(i,\lambda)}||S_{n-1}||C_{(j,\delta)}|}{|\mathcal{S}_n\times \mathcal{S}^{opp}_{n-1}|}\sum_{\theta\in \mathcal{P}_{n-k}}\mathcal{X}^{(k,\theta)}(1)\omega^{(k,\theta)}_{(i,\lambda)} \omega^{(k,\theta)}_{(j,\delta)} \overline{\omega^{(k,\theta)}_{(r,\rho)}} \nonumber \\
&=&\frac{|C_{(i,\lambda)}||C_{(j,\delta)}|}{n}\sum_{\theta\in \mathcal{P}_{n-k}}\frac{\mathcal{X}^{(k,\theta)}_{(i,\lambda)}\mathcal{X}^{(k,\theta)}_{(j,\delta)}\mathcal{X}^{(k,\theta)}_{(r,\rho)}}{\dim X^{(k,\theta)}\dim X^{(k,\theta)}},
\end{eqnarray}
where the second equation is due to Equation \eqref{fct_sph_zonal=chara}. In \cite{jackson2012character}, Jackson and Sloss were interested in the coefficient of ${\bf C}_{(r,\rho)}$ in the expansion of ${\bf C}_{(i,\lambda)}{\bf C}_{(j,\delta)}$ which is after simple computing nothing but $\frac{c^{(r,\rho)}_{(i,\lambda)(j,\delta)}}{(n-1)!}.$ Using equation \eqref{str_coef_doub_class_diag(S_n)}, one can re-obtain their result in \cite[Section 2.2.2]{jackson2012character}.
\bibliographystyle{alpha}
\bibliography{biblio}
\label{sec:biblio}

\end{document}